\documentclass[a4paper,12pt]{amsart}

\usepackage{amssymb}
\usepackage{amsmath}
\usepackage{amsthm}
\usepackage{mathtools}

\usepackage{enumerate}

\usepackage{latexsym}
\usepackage{mathrsfs}

\usepackage[colorlinks]{hyperref}
\usepackage{hyperref}
\renewcommand\eqref[1]{(\ref{#1})} 

\allowdisplaybreaks \numberwithin{equation}{section}

\theoremstyle{plain}
\newtheorem{theorem}{Theorem}[section]

\newtheorem{corollary}[theorem]{Corollary}

\theoremstyle{definition}
\newtheorem{definition}[theorem]{Definition}
\newtheorem{remark}[theorem]{Remark}
\newtheorem{example}[theorem]{Example}

\begin{document}

\title[Almost sectorial operators]{Direct and inverse abstract Cauchy problems with fractional powers of almost sectorial operators}

\author[J.E. Restrepo]{J.E. Restrepo}
\address{\textcolor[rgb]{0.00,0.00,0.84}{Joel E. Restrepo \newline Department of Mathematics \newline Center for Research and Advanced Studies \newline Av. IPN 2508, 07360, Mexico City, Mexico}}
\email{\textcolor[rgb]{0.00,0.00,0.84}{joel.restrepo@cinvestav.mx; cocojoel89@yahoo.es}}

\date{\today}
\thanks{The paper has been accepted for publication in the journal “Fractional Calculus and Applied Analysis”.}

\begin{abstract}
We derive the explicit solution operator of an abstract Cauchy problem involving a time-variable coefficient and a fractional power of an almost sectorial operator. The time-variable coefficient is recovered by solving the inverse abstract Cauchy problem using the solution operator representation. As a complement, we also study similar problems by considering almost sectorial operators that depend on a time-variable.  
\end{abstract}

\keywords{Direct abstract Cauchy problems, inverse abstract Cauchy problems, almost sectorial operators, time-variable coefficient, bounded and unbounded operators, dense and non-dense domains.} 

\subjclass[2010]{65J08; 65N21, 34K08, 47A05.}

\maketitle

\tableofcontents

\section{Introduction}

In functional analysis, specifically in operator theory, the Hille--Yosida theorem characterizes the generators of strongly continuous semigroups of linear operators on Banach spaces. That is, for a linear and densely defined operator $A:D(A)\subset X\to X$ ($X$ a complex Banach space), the following estimate holds:
\[
\|(\lambda I-A)^{-n}\|\leqslant \frac{C}{(\lambda-\omega)^n},\quad \lambda>\omega,
\]
for all positive integers $n$ and any $\lambda$ in the resolvent set $\rho(A).$ For contraction semigroups, the Hille--Yosida theorem can be easily characterized by those operators $A$ such that $\|(\lambda I-A)^{-1}\|\leqslant \frac{1}{\lambda},$ for any $\lambda\in\rho(A)$ with $\lambda>0.$ Nevertheless, there are different classes of important and interesting operators that do not satisfy the estimates above. We recall the so-called almost sectorial operators, see Definition \ref{aso}. More details about such operators can be found e.g. in \cite{r-13,JEE2002,r-18}. It is known that operators satisfying the condition \eqref{ineq-aso} do not generate a $C_0$-semigroup since they do not satisfy the Hille--Yosida estimates. But, in this case, the operator $A$ will generate a type of semigroup called an analytic semigroup of growth order $\gamma$ (see Definition \ref{analytic-s}). This concept was first introduced by Da Prato in \cite{r-2} for positive integer orders. The generalization for any positive order was given in different works e.g. \cite{r-13,oka,r-18,r-22} and the references therein. These semigroups are not strongly continuous at $t=0,$ and this is one of the main differences compared with the $C_0$-semigroups. Also, in this work, the operators do not need to have dense domains and/or ranges. This generality is not always assumed everywhere since some researchers studied almost sectorial operators defined over domains that are dense as well, see e.g. \cite{dense-1}. Almost sectorial operators frequently appear by the consideration of elliptic operators in regular spaces. Let us briefly recall some classical examples in this setting. For instance, the negative Laplacian in a bounded domain $\Omega$ is sectorial (this means $\gamma=1$ in Definition \ref{aso}) under some suitable boundary conditions in $L^p(\Omega)$ \cite[Section 1.3]{6-intro}. Moreover, it is also sectorial in spaces of bounded or continuous functions \cite{more-intro,haase,19-intro}. While, in the space of H\"older continuous functions a translation of the Laplacian is almost sectorial, see e.g. \cite{wahl} or \cite[Example 3.1.33]{lunardi}. Some other good examples of almost sectorial operators can be found in \cite[Section 2]{JEE2002}. An example of an almost sectorial operator that is not sectorial and involves Riemann-Liouville fractional derivatives in the Banach space $C_{\infty}(\mathbb{R};\mathbb{C}^2)$ (continuous functions that vanish at infinity) can be found in \cite[Example 3.2]{peri}. In \cite{peri}, more examples of almost sectorial operators are given using Riemann-Liouville derivatives.     
\\
\indent In this paper, we show explicitly the unique solution operator for the abstract Cauchy problem involving a time-dependent variable coefficient and a fractional power of an almost sectorial operator. By using the representation of the solution operator, we then consider two direct abstract Cauchy problems with the same time-variable coefficient, with the target of finding the time-dependent coefficient (inverse abstract Cauchy problems). To the best of our knowledge, it seems that it is the first time that the latter problems have been considered by means of almost sectorial operators. Here, we mainly follow some ideas from \cite{ORS} and \cite{JEE2002}. Note that the functional calculus developed for almost sectorial operators in \cite{JEE2002} will help us to construct the solution operators of our equations. This approach is totally different from the one used in \cite{ORS}. Also, in this paper, we consider ordinary time derivatives; but future works will be studied with time-fractional derivatives \cite{li1,li2}.        
\\
\indent We complement our studies, by considering the case of non-autonomous evolution equations with almost sectorial operators. Here, the operator depends on a time-variable. For some works in this direction, we recommend checking e.g. \cite{2008,non}. Again, for this case, we find the unique solution operator (explicitly) for the abstract Cauchy problem with a time-dependent variable coefficient and an almost sectorial operator $A(\tau)$ for $\tau\geqslant0$. We also study an inverse abstract Cauchy problem.
\\
\indent The structure of the paper is as follows. In Section \ref{preli}, we recall some definitions and results on almost sectorial operators and abstract evolution equations. We continue with Section \ref{direct}, where we give some of the main results on the representation of the solution operator of abstract Cauchy problems with time-variable coefficients. Section \ref{inverse} is devoted to the study of the inverse abstract Cauchy problem in finding the time-dependent variable coefficient. In the last Section \ref{Non-autonomous}, we also study non-autonomous direct and inverse  Cauchy problems. In this case, almost sectorial operators depend on a time-variable. 

\section{Preliminaries}\label{preli}

We begin by recalling some definitions and results that will be used throughout the paper. We always assume that $(X,\|\cdot\|)$ is a complex Banach space. An operator in $X$ means a linear map $A:D(A)\subset X\to X$ whose domain $D(A)$ is a linear subspace of $X.$ We denote by $\sigma(A)$ the spectrum of $A,$ and by $R(A)$ its range. For $0<\mu<\pi,$ we denote by $S_\mu^0$, the open sector $\{z\in\mathbb{C}\setminus\{0\}: |\text{arg}\,z|<\mu\},$ and its closure $S_\mu:=\{z\in\mathbb{C}\setminus\{0\}: |\text{arg}\,z|\leqslant\mu\}\cup \{0\}.$ We consider the function $\text{arg}$ with values in $(-\pi,\pi].$  Set 
\[
\mathcal{F}_0^{\gamma}(S_\mu^0)=\bigcup_{s<0}\Psi_s^{\gamma}(S_\mu^0)\cup\Psi_0(S_\mu^0),
\]
and 
\[
\mathcal{F}(S_\mu^0)=\{f\in\mathcal{H}(S_\mu^0):\,\,\text{there exists}\,\,k,n\in\mathbb{N}\,\, \text{such that}\,\, f\psi_n^k\in \mathcal{F}_0^{\gamma}(S_\mu^0)\},
\]
where 
\[
\mathcal{H}(S_\mu^0)=\{f:S_\mu^0\to\mathbb{C};\,\,f\,\,\text{is holomorphic}\},
\]
\[
\mathcal{H}^{\infty}(S_\mu^0)=\{f\in \mathcal{H}(S_\mu^0),\,\,f\,\,\text{is bounded}\},
\]
\[
\varphi_0(z)=\frac{1}{1+z},\quad \psi_n(z)=\frac{z}{(1+z)^n},\quad z\in\mathbb{C}\setminus\{-1\},\,\,n\in\mathbb{N}\cup\{0\},
\]
\[
\Psi_0(S_\mu^0)=\left\{f\in\mathcal{H}(S_\mu^0):\sup_{z\in S_\mu^0}\left|\frac{f(z)}{\varphi_0(z)}\right|<+\infty\right\},
\]
and for each $s<0,$
\[
\Psi_s^{\gamma}(S_\mu^0)=\left\{f\in\mathcal{H}(S_\mu^0):\sup_{z\in S_\mu^0}|\psi_n^s(z)f(z)|<+\infty\right\},
\]
where $n$ is the smallest integer such that $n\geqslant2$ and $\gamma+1<-(n+1)s.$ It follows
\[
\mathcal{F}_0^\gamma(S_\mu^0)\subset \mathcal{H}^{\infty}(S_\mu^0)\subset \mathcal{F}(S_\mu^0)\subset \mathcal{H}(S_\mu^0), 
\]
and for $k,n\in\mathbb{N}\cup\{0\}$ with $n>k$, one has $\psi_n^k\in \mathcal{F}_0^{\gamma}(S_\mu^0).$
\begin{definition}\label{aso}
Let $-1<\gamma<0$ and $0\leqslant \omega<\pi.$ By $\Theta_\omega^\gamma(X)$ we denote the set of all closed linear operators $A:D(A)\subset X\to X$ which satisfy
\begin{enumerate}[(a)]
    \item $\sigma(A)\subset S_\omega.$
    \item\label{ineq-aso} For any $\omega<\mu<\pi$, there exists a positive constant $C_\mu$ such that  
    \[
    \|(z-A)^{-1}\|\leqslant C_{\mu}|z|^{\gamma},\quad\text{for any}\quad z\notin S_\mu.
    \]
\end{enumerate}
\end{definition}

\indent For simplicity, we write $\Theta_\omega^\gamma$ instead of $\Theta_\omega^\gamma(X)$. {\it A linear operator $A$ will be called an almost sectorial operator in $X$ if $A\in \Theta_\omega^\gamma.$} Note that $0\in\rho(A)$ for any $A\in\Theta_\omega^\gamma.$ Also, $A$ is injective \cite[Remark 2.2]{JEE2002}. The operators in $\Theta_\omega^\gamma$ have the possibility of having non-dense domain and/or range. This latter feature gives a different panorama with respect to the classical results, where dense domains are assumed \cite{pazy,Prus}.    
\\
\indent We usually denote by $\Gamma_\theta$ $(0<\theta<\pi)$ the path 
\[\{\rho e^{-i\theta}:\rho>0\}\cup\{\rho e^{i\theta}:\rho>0\}\]
oriented such that the sector $S_\theta^0$ lies to the left of $\Gamma_\theta.$
\begin{definition}
Let $A\in\Theta_\omega^\gamma,$ $\omega<\theta<\mu<\pi,$ and $f\in\mathcal{F}(S_\mu^0).$ For $k,m\in\mathbb{N}$ such that $f\psi_m^k\in\mathcal{F}_0^\gamma(S_\mu^0),$ we define the linear operator $f(A)$ in the domain $D(f(A))=\{x\in X: (f\psi_m^k)(A)x\in D(A^{(m-1)k})\}$ by
\[
f(A)=(\psi_m^k(A))^{-1}\frac{1}{2\pi i}\int_{\Gamma_\theta}(f\psi_m^k)(z)(z-A)^{-1}{\rm d}z,
\]
where the above integral is absolutely convergent and defines a bounded linear operator on $X.$
\end{definition}

\indent For every $\beta\in\mathbb{C},$ the function $h_\beta(z)=z^{\beta}$ for all $z\in\mathbb{C}\setminus(-\infty,0]$ belongs to the class $\mathcal{F}(S_\nu^0)$ for $0<\nu<\pi,$ and hence we can define $A^{\beta}=h_\beta(A).$ Some of the main properties of $A^{\beta}$ ($A\in\Theta_\omega^\gamma,$\,$\beta\in\mathbb{C}$) are the following \cite[Theorem 3.2]{JEE2002}:
\begin{enumerate}[(i)]
    \item The operator $A^{\beta}$ is closed.
    \item $A^{\beta}A^{\beta^*}\subset A^{\beta+\beta^*},$ for $\beta^*\in\mathbb{C}.$ If $D(A^{\beta+\beta^*})\subset D(A^{\beta^*}),$ then $A^{\beta}A^{\beta^*}=A^{\beta+\beta^*}.$
    \item $A^{\beta}$ is injective and $(A^{\beta})^{-1}=A^{-\beta}.$
\end{enumerate}

\indent Take $0<\alpha<\frac{\pi}{2\omega}$ and $t\in S_{\frac{\pi}{2}-\alpha\omega}^0$, with $\omega<\mu<\text{min}\left\{\pi,\frac{\pi-2|\text{arg}\,t|}{2\alpha}\right\}.$ Suppose that the function $g_{\alpha,t}(z)=e^{-tz^{\alpha}}$ is defined on $z\in S_\mu^0.$ Assume also that $A\in \Theta_\omega^\gamma.$ From \cite[Lemma 2.13]{JEE2002}, we conclude that 
\begin{equation}\label{representation}
\mathscr{T}_\alpha(t)=\frac{1}{2\pi}\int_{\Gamma_\theta}e^{-tz^{\alpha}}(z-A)^{-1}{\rm d}z,\quad \omega<\theta<\mu,
\end{equation}
is a bounded linear operator in $X.$ In particular, the above holds for $0<\alpha\leqslant 1.$ Note that the above expression for $\mathscr{T}_\alpha(t)$ could also be denoted by $e^{-tA^{\alpha}}.$ The latter notation is frequently used for strongly continuous semigroups generated by $A^{\alpha}$. For almost sectorial operators, $A^{\alpha}$ does not always generate a strongly continuous semigroup, see item (5) of Theorem \ref{thm3.9}. This semigroup is singular at $t=0,$ see item (2) of Theorem \ref{thm3.9}.   

\indent The following definition was introduced in \cite{r-2}. Analytic semigroups of growth order are defined. The notion is adapted from the $C_0-$ semigroups due to the semigroup property holding for these kinds of operators.

\begin{definition}\label{analytic-s}
Let $0<\mu<\pi/2$ and $\kappa>0.$ A family $\{\mathscr{T}(t):t\in S_\mu^0\}$ is said to be an analytic semigroup of growth order $\kappa$ if the following conditions hold:
\begin{enumerate}[(I)]
    \item $\mathscr{T}(t+s)=\mathscr{T}(t)\mathscr{T}(s)$ for any $t,s\in S_\mu^0.$
    \item \label{ii}The mapping $t\to\mathscr{T}(t)$ is analytic in $S_\mu^0.$
    \item There exists a positive constant $C$ such that 
    \[
    \|\mathscr{T}(t)\|\leqslant Ct^{-\kappa},\quad\text{for any}\quad t>0.
    \]
    \item If $\mathscr{T}(t)x=0$ for some $t\in S_\mu^0$, then $x=0.$
\end{enumerate}
\end{definition}

\indent The above definition has two differences with respect to the one given by Da Prato. First, the set $\displaystyle X_0=\bigcup_{t>0}\mathscr{T}(t)X$ is not necessarily dense in $X.$ Second, the strong continuity of the mapping $t\to \mathscr{T}(t)$ for $t>0$ is replaced by condition \eqref{ii}.

\medskip\indent We recall some properties of the families $\{\mathscr{T}_{\alpha}(t):t\in S_{\frac{\pi}{2}-\alpha\omega}\}$ stated in \cite[Theorem 3.9]{JEE2002} that will be used implicitly in some of our proofs. 

\begin{theorem}\label{thm3.9}
Suppose that $A\in\Theta_\omega^\gamma$ for some $0<\omega<\frac{\pi}{2}$ and $0<\alpha<\frac{\pi}{2\omega}.$ Then the family $\{\mathscr{T}_\alpha(t):t\in S_{\frac{\pi}{2}-\alpha\omega}^0\}$ is an analytic semigroup of growth order $\frac{1+\gamma}{\alpha}.$ So the following assertions are true.
\begin{enumerate}
    \item $\mathscr{T}_\alpha(t+s)=\mathscr{T}_\alpha(t)\mathscr{T}_\alpha(s)$ for any $t,s\in S_{\frac{\pi}{2}-\alpha\omega}.$
    \item There exists a positive constant $C(\gamma,\alpha)$ such that 
    \[
    \|\mathscr{T}_\alpha(t)\|\leqslant Ct^{-\frac{\gamma+1}{\alpha}},\quad\text{for any}\quad t>0.
    \]
    \item\label{need} The range $R(\mathscr{T}_\alpha(t))$ of $\mathscr{T}_\alpha(t)$ with $t\in S_{\frac{\pi}{2}-\alpha\omega}^0$, is contained in $D(A^{\infty})$. In particular, $R(\mathscr{T}_\alpha(t))\subset D(A^{\beta})$ for all $\beta\in\mathbb{C}$ with $\text{Re}\,\beta>0,$ and
    \[
A^{\beta}\mathscr{T}_\alpha(t)x=\frac{1}{2\pi i}\int_{\Gamma_\theta}z^{\beta}e^{-tz^{\alpha}}(z-A)^{-1}x\,{\rm d}z,\quad\text{for all}\quad x\in X.
    \]
    Also, there is a positive constant $C(\gamma,\alpha,\beta)$ such that 
    \[
    \|A^{\beta}\mathscr{T}_\alpha(t)\|\leqslant t^{-\frac{\gamma+1+{\rm Re}\,\beta}{\alpha}},\quad\text{for all}\quad t>0.
    \]
    \item\label{deri} The function $t\mapsto \mathscr{T}_{\alpha}(t)$ is analytic in $S_{\frac{\pi}{2}-\alpha\omega}^0$ and
    \[
    \frac{{\rm d}^k}{{\rm d}t^k}\mathscr{T}_{\alpha}(t)=(-1)^k A^{k\alpha}\mathscr{T}_\alpha(t),\quad\text{for all}\quad t\in S_{\frac{\pi}{2}-\alpha\omega}^0.
    \]
    \item\label{continuity} Let $\Omega_\alpha=\left\{x\in X:\displaystyle\lim_{t\to0,\,t>0}\mathscr{T}_\alpha(t)x=x\right\}$ be the continuity set of $\mathscr{T}_\alpha(\cdot)$. If $\beta>1+\gamma$ then $D(A^{\beta})\subset\Omega_\alpha.$
\end{enumerate}
\end{theorem}

\begin{remark}
It is important to highlight and stress that, in general, the analytic semigroups of some growth order are not strongly continuous at $t=0,$ see e.g. Definition \ref{analytic-s} and Theorem \ref{thm3.9} (item 5).   
\end{remark}

\medskip\indent {\bf Abstract evolution equations.} Suppose that $A\in \Theta_\omega^\gamma$ with $0<\omega<\pi/2$ and $-1<\gamma<0.$ Consider the following abstract Cauchy problem:
\begin{equation}\label{acp}
\begin{split}
\partial_t u(t):=u_t(t)&=A^{\alpha}u(t),\quad 0<t<T,\quad 0<\alpha<\frac{\pi}{2\omega}, \\
u(0)&=u_0\in D(A^{\alpha}).
\end{split}
\end{equation}
By a classical solution of problem \eqref{acp}, we mean a function 
\[
u\in C\big([0,T);X\big)\cap C^1\big((0,T);X\big)
\]
which takes values in $D(A^{\alpha})$ for any $0<t<T$ and satisfies \eqref{acp}. It is important to note that if $u_0\in D(A^{\alpha})$ and $\alpha>1+\gamma$ then there exists a classical solution of \eqref{acp} (Theorem \ref{thm3.9}, item \eqref{continuity}). It can be proved (see e.g. \cite[Chapter 4]{JEE2002}) that equation \eqref{acp} has a unique classical solution given by 
\begin{equation}\label{acpsol}
u(t)=\mathscr{T}_\alpha(t)u_0,\quad\text{for any}\quad 0<t<T.
\end{equation} 

\section{Direct abstract Cauchy problem}\label{direct}

In this section, we study the direct abstract Cauchy problem with a time-dependent variable coefficient $\phi$ and an almost sectorial operator. We consider the abstract  Cauchy problem in a complex Banach space $X$ with the coefficient $\phi(t)$ and $A\in \Theta_\omega^\gamma$ $\big(0<\omega<\pi/2\big)$ as follows:
\begin{equation}\label{eqc}
\begin{split}
u_t(t)&=\phi(t)A^{\alpha}u(t),\quad 0<t<T,\quad 1+\gamma<\alpha<\frac{\pi}{2\omega}, \\
u(0)&=u_0\in D(A^{\alpha}),
\end{split}
\end{equation} 
where the time-dependent function $\phi:[0,+\infty)\to[0,+\infty)$ is continuous such that $\phi(t)>0$ for $t>0.$

\indent Below we prove that the solution operator $\mathscr{T}_{\alpha,\phi}(t)$ of problem \eqref{eqc} has an explicit representation in terms of the coefficient $\phi(t)$ and the analytic semigroup $\mathscr{T}_\alpha(t)$ of growth order $\frac{\gamma+1}{\alpha}$. 

\begin{theorem}\label{directsolution}
Let $X$ be a complex Banach space. Then the solution operator of problem \eqref{eqc} is given by:
\begin{equation}\label{solutioneqc}
\mathscr{T}_{\alpha,\phi}(t)=\mathscr{T}_\alpha\left\{\left(\int_0^t \phi(s){\rm d}s\right)\right\},\quad 0<t<T.
\end{equation}
\end{theorem}
\begin{proof}
Since $\phi(t)>0$ for $t>0$ and \eqref{representation}, it follows
\[
\mathscr{T}_{\alpha,\phi}(t)=\frac{1}{2\pi}\int_{\Gamma_\theta}e^{-z^{\alpha}\int_0^t \phi(s){\rm d}s}(z-A)^{-1}{\rm d}z,\quad \omega<\theta<\mu<\text{min}\left\{\pi,\frac{\pi}{2\alpha}\right\},
\]
is a bounded linear operator in $X.$ Now, by Theorem \ref{thm3.9} (item (4)), we obtain
\begin{align*}
\frac{{\rm d}}{{\rm d}t}\mathscr{T}_{\alpha,\phi}(t)=\phi(t)A^{\alpha}\mathscr{T}_{\alpha,\phi}(t).    
\end{align*}
From Theorem \ref{thm3.9} (item \eqref{continuity}) and $\alpha>1+\gamma$, $u(t)=\mathscr{T}_{\alpha,\phi}(t)u_0$ is a classical solution of problem \eqref{eqc}.
\end{proof}

\indent Let us now mention, in particular, the classical and important case $\alpha=1.$ Thus, suppose that:
\begin{equation}\label{eqc-coro}
\begin{split}
u_t(t)&=\phi(t)Au(t),\quad 0<t<T, \\
u(0)&=u_0\in D(A),
\end{split}
\end{equation}
where $A\in \Theta_\omega^\gamma$ with $0<\omega<\pi/2.$ In what follows, we denote by $\mathscr{T}(t),$ $\mathscr{T}_{\phi}(t)$, the operators $\mathscr{T}_1(t)$, and $\mathscr{T}_{1,\phi}(t)$ respectively. 

\begin{corollary}\label{directsolution-coro}
Let $X$ be a complex Banach space. Then the solution operator of problem \eqref{eqc-coro} can be obtained as:
\[
\mathscr{T}_{\phi}(t)=\mathscr{T}\left\{\left(\int_0^t \phi(s){\rm d}s\right)\right\},\quad 0<t<T.
\]
\end{corollary}

\section{Inverse abstract Cauchy problems}\label{inverse}

We now use the representation of the solution operator of the (direct) abstract Cauchy problem (found in the previous Section) to find the time-dependent coefficient for the inverse abstract Cauchy problem. The method is based on the consideration of two (direct) abstract Cauchy problems with the same time-dependent coefficient.  

\indent We start by establishing a useful property that involves the analytic semigroup of growth order and the fractional power of an operator. Let $A\in \Theta_\omega^\gamma$ for some $0<\omega<\pi/2.$ From \cite[Theorem 2.5]{JEE2002} and Theorem \ref{thm3.9} (item \eqref{need}), we have
\begin{equation}\label{conmutative}
A^{\alpha}\mathscr{T}_\alpha(t)x=\mathscr{T}_\alpha(t)A^{\alpha}x,\quad t\in S^0_{\frac{\pi}{2}-\alpha\omega},\quad x\in D(A^{\alpha}),\quad 0<\alpha<\frac{\pi}{2\omega},
\end{equation}
since the function $z^{\alpha}e^{-tz^{\alpha}}$ is in the class $\mathcal{F}_0^\gamma(S_\mu^0)$ $\big(\omega<\theta<\mu<\text{min}\left\{\pi,\frac{\pi}{2\alpha}\right\}\big).$ More details about it can be found in \cite[Section 2]{JEE2002}. 

\indent For $A\in \Theta_\omega^\gamma$ $(0<\omega<\pi/2)$ and $\phi:[0,+\infty)\to[0,+\infty)$ being a continuous function with $\phi(t)>0$ for $t>0$, we consider the following two abstract Cauchy problems with the variable coefficient $\phi$:
\begin{equation}\label{eqce1}
\begin{split}
u_t(t)&=\phi(t)A^{\alpha}u(t),\quad 0<t<T,\quad 1+\gamma<\alpha<\frac{\pi}{2\omega}, \\
u(0)&=u_0\in D(A^{\alpha}),
\end{split}
\end{equation} 
and 
\begin{equation}\label{eqce2}
\begin{split} v_t(t)&=\phi(t)A^{\alpha}v(t),\quad 0<t<T,\quad 1+\gamma<\alpha<\frac{\pi}{2\omega}, \\
v(0)&=A^{\alpha}u_0\in D(A^{\alpha}).
\end{split}
\end{equation} 

\indent From now on, we always assume $u_0\neq0$ (respectively $Au_0\neq0$). Note that if $u_0=0$ then the classical solution $u$ of \eqref{eqce1} (respectively \eqref{eqce2}) satisfies $u\equiv0$ for $0<t<T.$ These latter cases are not considered to recover the variable coefficient $\phi$ above.

\indent By Theorem \ref{directsolution}, the solutions of problems \eqref{eqce1} and \eqref{eqce2} are given, respectively, by 
\begin{equation}\label{solutionse}
u(t)=\mathscr{T}_\alpha\left\{\left(\int_0^t \phi(r){\rm d}r\right)\right\}u_0,\,\text{and}\,v(t)=\mathscr{T}_\alpha\left\{\left(\int_0^t \phi(r){\rm d}r\right)\right\}A^{\alpha}u_0.
\end{equation}

\indent Let us now fix an observation point $q$ for the following two time dependent quantities:
\[h_1(t):=u(t,q)\neq0,\quad h_2(t):=v(t,q)\neq0, \quad 0<t<T.\]

\indent This point is not an element itself of the Banach space $X.$ In fact, the point belongs to the measure space where we are defining $X.$ For instance, consider the Banach space $C^{l}(\overline{\Omega})$ for $0<l<1$ and $\Omega$ a bounded domain in $\mathbb{R}^n$ $(n\geqslant1)$ with a smooth boundary. In this case, we take $q\in\Omega.$ While, in the space $L^3(\mathbb{R}^2),$ we then take $q\in\mathbb{R}^2.$ In the latter spaces we can always find and define different types of almost sectorial operators, see e.g. \cite[Section 2]{JEE2002} or \cite[Section 6]{section3}.

\begin{example}
Let $\Omega$ be a bounded domain in $\mathbb{R}^n$ $(n\geqslant1)$ with boundary $\partial\Omega$ of class $C^4.$ Let $X=C^{\beta}(\Omega)$ $(0<\beta<1).$ Set 
\[
\Delta=\sum_{k=1}^n \frac{\partial^2}{\partial x_k^2},\quad D(\Delta)=\{u\in C^{2+\beta}(\Omega):\,\,u(0)=0\,\,\text{on}\,\,\partial\Omega\}.
\]

From \cite[Example 2.3]{JEE2002}, it follows that there exist $\nu,\rho>0,$ such that 
\[
\Delta+\nu\in\Theta_{\frac{\pi}{2}-\rho}^{\frac{\beta}{2}-1}\big(C^{\beta}(\Omega)\big).
\]

Consider the following initial value problem for the heat equation on $\mathbb{R}^n:$
\begin{align}\label{heate}
\begin{split}
w_t(t,x)&=\phi(t)\big(\Delta_x +\nu\big)w(t,x),\quad x\in\Omega\subset\mathbb{R}^n,\,\,t>0, \\
w(t,x)&|_{_{_{t=0+}}}=g(x),\quad x\in\Omega, \\
w(t,x)&|_{_{_{t=0+}}}=0,\quad x\in\partial\Omega,
\end{split}
\end{align}   
where $\text{supp}(g)\subset\Omega,$ $g,\widehat{g}\in L^1(\mathbb{R}^n)$. If $g\in D(\Delta)$ then there exists a classical solution of problem \ref{heate} (see Theorem \ref{directsolution}). The solution can be also found by applying the space Fourier transform and then solving the transform equation, which is an ordinary differential equation with respect to the variable $t$. So, the solution is given by:
\begin{equation}\label{solutionhe}
w(t,x)=\frac{e^{\nu \int_0^t \phi(s){\rm d}s}}{\left(4\pi\int_0^t \phi(s){\rm d}s\right)^{n/2}}\int_{\Omega}e^{-\frac{|x-y|^2}{\int_0^t \phi(s){\rm d}s}}g(y)dy,\quad t>0,\,\, x\in\Omega. 
\end{equation} 

Note that from representation \eqref{solutionhe} it is clear that we have many possibilities for $q\in\Omega$ such that $h_1(t)=w(t,q)\neq0$ as well as $h_2(t)\neq0.$
\end{example}

\medskip\indent We establish the main result of this section as follows. 

\begin{theorem}
Let $X$ be a complex Banach space. Then there exists a unique function $\phi(t)$ of the inverse abstract Cauchy problem \eqref{eqce1} (and \eqref{eqce2}) given by 
 \[\phi(t)=\frac{\partial_t h_1(t)}{h_2(t)},\quad\text{for any}\quad 0<t<T.\]
\end{theorem}
\begin{proof}
Using \eqref{solutionse}, Theorem \ref{thm3.9} (item \eqref{deri}) and property \eqref{conmutative} we obtain
\begin{align*}
\partial_t h_1(t)=\partial_t u(t,q)&=\partial_t \mathscr{T}_\alpha\left\{\left(\int_0^t \phi(r){\rm d}r\right)\right\}u_0(t,q) \\
&=\phi(t)A^{\alpha}\mathscr{T}_\alpha\left\{\left(\int_0^t \phi(r){\rm d}r\right)\right\}u_0(t,q) \\
&=\phi(t)\mathscr{T}_\alpha\left\{\left(\int_0^t \phi(r){\rm d}r\right)\right\}A^{\alpha}u_0(t,q) \\
&=\phi(t)v(t,q)=\phi(t)h_2(t),
\end{align*}
which completes the proof.
\end{proof}

\indent Again, it is worth highlighting the particular case $\alpha=1.$ Hence, let us study the following two abstract Cauchy problems with the variable coefficient $\phi$:
\begin{equation}\label{eqce1-coro}
\begin{split}
u_t(t)&=\phi(t)Au(t),\quad 0<t<T, \\
u(0)&=u_0\in D(A),
\end{split}
\end{equation} 
and 
\begin{equation}\label{eqce2-coro}
\begin{split}
v_t(t)&=\phi(t)Av(t),\quad 0<t<T, \\
v(0)&=Au_0\in D(A),
\end{split}
\end{equation}
where $A\in \Theta_\omega^\gamma$ $(0<\omega<\pi/2)$ and $\phi:[0,+\infty)\to[0,+\infty)$ is a continuous function with $\phi(t)>0$ for $t>0.$ The solution operators of problems \eqref{eqce1-coro} and \eqref{eqce2-coro} are given respectively by 
\[
u(t)=\mathscr{T}\left\{\left(\int_0^t \phi(r){\rm d}r\right)\right\}u_0,\quad\text{and}\quad v(t)=\mathscr{T}\left\{\left(\int_0^t \phi(r){\rm d}r\right)\right\}Au_0.
\]
Here, for the fixed point $q$, we denote by
\[
\widetilde{h_1}(t)=\mathscr{T}\left\{\left(\int_0^t \phi(r){\rm d}r\right)\right\}u_0(t,q)\,\,\text{and}\,\,\widetilde{h_2}(t)=\mathscr{T}\left\{\left(\int_0^t \phi(r){\rm d}r\right)\right\}Au_0(t,q).
 \]
\begin{corollary}
Let $X$ be a complex Banach space. Assume  $\phi:[0,+\infty)\to[0,+\infty)$ is a continuous function such that $\phi(t)>0$ for $t>0$. Then there exists a unique function $\phi(t)$ of the inverse abstract Cauchy problem \eqref{eqce1-coro} (and \eqref{eqce2-coro}) constructed as 
 \[\phi(t)=\frac{\partial_t \widetilde{h_1}(t)}{\widetilde{h_2}(t)},\quad\text{for any}\quad 0<t<T.\] 
\end{corollary}

\section{Non-autonomous Cauchy problems}\label{Non-autonomous}

In this section, we suppose that our linear operator $A$ depends also on a time-variable $\tau\geqslant0$. In this case, the operator is usually denoted by $A(\tau).$ We will work with the following class of operators \cite[Definition 2.3]{non} (see also \cite{2008}) which we also denote by $\Theta_\omega^\gamma(X)$ because it is adapted from Definition \ref{aso}. 
\begin{definition}\label{aso-n}
For each $\tau\geqslant0.$ Let $-1<\gamma<0$ and $0\leqslant \omega<\pi.$ By $\Theta_\omega^\gamma(X)$ we denote the set of all closed linear operators $A(\tau):D(A(\tau))\subset X\to X$ that satisfy
\begin{enumerate}[(a)]
    \item $\sigma(A(\tau))\subset S_\omega.$
    \item For any $\omega<\mu<\pi$, there exists a positive constant $C_\mu$ independent of $\tau$ such that  
    \[
    \|(z-A(\tau))^{-1}\|\leqslant C_{\mu}|z|^{\gamma},\quad\text{for any}\quad z\notin S_\mu.
    \]
\end{enumerate}
\end{definition}

\indent A linear operator $A(\tau)$ will be called an almost sectorial operator in $X$ if $A(\tau)\in\Theta_\omega^\gamma(X)$ for each $\tau\geqslant0.$ In this scenario, we can also generate a type of semigroup with $A(\tau).$ In fact, by \cite[Lemma 2.7]{non} (see also \cite[Pages 634--637]{2008}), the semigroup $\{T_{A(\tau)}(t)\}_{t\in S_{\frac{\pi}{2}-\omega}}^0$ $(\omega<\theta<\mu<\frac{\pi}{2}-|\text{arg}\,t|,\,\text{for each}\,\tau\geqslant0)$ of growth order $1+\gamma$ is studied and defined as follows 
\[
T_{A(\tau)}(t)=\frac{1}{2\pi i}\int_{\Gamma_\theta}e^{-tz}(z-A(\tau))^{-1}\,{\rm d}z,\quad t\in S_{\frac{\pi}{2}-\omega}^0, 
\]
which is a bounded linear operator in $X,$ such that: i) $T_{A(\tau)}(0)=I;$ ii) $T_{A(\tau)}(t+s)=T_{A(\tau)}(t)T_{A(\tau)}(s)$ for each $t,s>0;$ iii) $T_{A(\tau)}(t)x=0$ implies $x=0$ for $t>0;$ iv) $\lim_{t\to0}t^{1+\gamma}T_{A(\tau)}(t)$ is bounded; (v) the function $t\mapsto T_{A(\tau)}(t)$ is analytic in $S_{\frac{\pi}{2}-\omega}^0$ and $\frac{d^n}{dt^n}T_{A(\tau)}(t)=\big(-A(\tau)\big)^n T_{A(\tau)}(t)$ for all $t\in S_{\frac{\pi}{2}-\omega}^0.$

\medskip\indent Consider the following abstract Cauchy problem:
\begin{equation}\label{acp-t}
\begin{split}
u_t(t)&=A(\tau)u(t),\quad t>0,\quad \tau\geqslant0, \\
u(0)&=u_0\in X.
\end{split}
\end{equation}
By a classical solution of problem \eqref{acp-t}, we mean a function 
\[
u\in C\big([0,+\infty);X\big)\cap C^1\big((0,+\infty);X\big)
\]
that takes values in $D(A(\tau))$ for any $t>0$ and satisfies \eqref{acp-t}. It is known (see e.g. \cite[Lemma 2.3]{2008} or \cite[Lemma 2.7]{non}) that the equation \eqref{acp-t} has a unique classical solution given by 
\begin{equation}\label{acpsol-t}
u(t)=T_{A(\tau)}(t)u_0,\quad\text{for any}\quad t\geqslant0,\quad \tau\geqslant0.
\end{equation}

\subsection{ Direct abstract Cauchy problem}\label{direct-t}

\medskip We study the abstract  Cauchy problem with a coefficient $\phi(t)$ and $A(\tau)\in \Theta_\omega^\gamma$ $\big(0<\omega<\pi/2,\,\tau\geqslant0\big).$ Here, we assume that $\phi:[0,+\infty)\to[0,+\infty)$ is a continuous function such that $\phi(t)>0$ for $t>0.$ The problem to be studied reads as follows:
\begin{equation}\label{eqc-t}
\begin{split}
u_t(t)&=\phi(t)A(\tau)u(t),\quad t>0,\quad \tau\geqslant0, \\
u(0)&=u_0\in X.
\end{split}
\end{equation}  
The proof of the next statement follows the same calculations and steps as the proof of Theorem \ref{directsolution}. Therefore, we omit the proof and leave it to the interested reader.
\begin{theorem}\label{directsolution-t}
Let $X$ be a complex Banach space. Then the solution operator of the problem \eqref{eqc-t} is expressed as:
\begin{equation}\label{solutioneqc-t}
T_{A(\tau),\phi}(t)=T_{A(\tau)}\left\{\left(\int_0^t \phi(s){\rm d}s\right)\right\},\quad t\geqslant0,\quad \tau\geqslant0.
\end{equation}
\end{theorem}

\subsection{Inverse abstract Cauchy problems}\label{inverse-t}

We start by establishing a property that involves the semigroup $T_{A(\tau),\phi}(t)$ of growth order $1+\gamma$ and the operator $A(\tau)$. Let $A(\tau)\in \Theta_\omega^\gamma$ for some $0<\omega<\pi/2$ and $\tau\geqslant0.$ From \cite[Pages 22-24]{2008} (or \cite[Lemma 2.7 (ii)]{non}), we know that
\begin{equation}\label{conmutative-t}
A(\tau)T_{A(\tau),\phi}(t)x=T_{A(\tau),\phi}(t)A(\tau)x,\quad t\in S^0_{\frac{\pi}{2}-\omega},\,\, x\in D(A(\tau)),\,\,\tau\geqslant0.
\end{equation} 

\indent For $A(\tau)\in \Theta_\omega^\gamma$ $(0<\omega<\pi/2,\,\tau\geqslant0)$ and $\phi:[0,+\infty)\to[0,+\infty)$ being a continuous function with $\phi(t)>0$ for $t>0$, we consider the following two abstract Cauchy problems with the variable coefficient $\phi$:
\begin{equation}\label{eqce1-t}
\begin{split}
u_t(t)&=\phi(t)A(\tau)u(t),\quad t>0,\quad \tau\geqslant0, \\
u(0)&=u_0\in D(A(\tau)),
\end{split}
\end{equation} 
and 
\begin{equation}\label{eqce2-t}
\begin{split}
v_t(t)&=\phi(t)A(\tau)v(t),\quad t>0,\quad \tau\geqslant0, \\
v(0)&=A(\tau)u_0\in D(A(\tau)).
\end{split}
\end{equation} 
By Theorem \ref{directsolution-t}, the solution operators of the problems \eqref{eqce1-t} and \eqref{eqce2-t} are represented by
\begin{equation}\label{solutionse-t}
u(t)=T_{A(\tau)}\left\{\left(\int_0^t \phi(r){\rm d}r\right)\right\}u_0,\,v(t)=T_{A(\tau)}\left\{\left(\int_0^t \phi(r){\rm d}r\right)\right\}A(\tau)u_0.
\end{equation}

\indent Let us now fix an observation point $q$ for the following two time dependent quantities:
\[k_1(t):=u(t,q)\neq0,\quad k_2(t):=v(t,q)\neq0, \quad 0<t<T.\]

\indent We now provide the main result of this subsection. 
\begin{theorem}
Let $X$ be a complex Banach space. Then there exists a unique function $\phi(t)$ of the inverse abstract Cauchy problem \eqref{eqce1-t} (and \eqref{eqce2-t}) given by 
 \[\phi(t)=\frac{\partial_t k_1(t)}{k_2(t)},\quad\text{for any}\quad t>0.\]
\end{theorem}
\begin{proof}
Using \eqref{solutionse-t}, \cite[Lemma 2.3]{2008} (or \cite[Lemma 2.7 (i)]{non}) and property \eqref{conmutative-t} it follows that
\begin{align*}
\partial_t k_1(t)=\partial_t u(t,q)&=\partial_t T_{A(\tau)}\left\{\left(\int_0^t \phi(r){\rm d}r\right)\right\}u_0(t,q) \\
&=\phi(t)A(\tau)T_{A(\tau)}\left\{\left(\int_0^t \phi(r){\rm d}r\right)\right\}u_0(t,q) \\
&=\phi(t)T_{A(\tau)}\left\{\left(\int_0^t \phi(r){\rm d}r\right)\right\}A(\tau)u_0(t,q) \\
&=\phi(t)v(t,q)=\phi(t)k_2(t),
\end{align*}
which finishes the proof.
\end{proof}

\begin{remark}
Note that the equations in this section depend on operators $A(\tau)$ for some $\tau\geqslant0.$ Therefore, we study, in principle, an evolution equation with respect to the time-variable $t$ and an almost sectorial operator $A(\tau),$ where $t$ and $\tau$ do not have to be related. The case $A(t)$ (we mean, e.g. $u_t(t)=A(t)u(t);$ $u(\tau)=u_0;$ $\tau>t$) is totally different since the solution operator is associated with a new notion of a linear evolution process introduced in \cite{2008}. We do not study this case because this will be done elsewhere.
\end{remark}

\section*{\small
 Conflict of interest} 

 {\small
 The author declares that there are no conflicts of interest.}

\end{document}